\documentclass[12pt]{amsart}
\usepackage{setspace}
\usepackage{mathrsfs,amssymb,multirow,setspace,color}
\usepackage{hyperref}
\usepackage[a4paper, total={7in, 9in}]{geometry}
\theoremstyle{plain}
\usepackage{graphicx}
\newtheorem{theorem}{Theorem}[section]
\newtheorem{lemma}[theorem]{Lemma}
\newtheorem{proposition}[theorem]{Proposition}
\newtheorem{corollary}[theorem]{Corollary}

\theoremstyle{definition}

\theoremstyle{remark}

\input xy
\xyoption{all}

\def\c{{\Gamma(G)}}

\def\m{\mathcal{M}(G)}
\def\mf{\mathfrak{m}_d(G)}
\def\cy{\mathrm{Cyc}(G)}

\begin{document}
\title[On Hamiltonicity and Perfect Codes in Non-Cyclic Graphs of Finite Groups]{On Hamiltonicity and Perfect Codes in Non-Cyclic Graphs of Finite Groups}
\author[Parveen]{$\text{PARVEEN}^{^*}$ }
\author[B. Bhattacharjya]{Bikash Bhattacharjya}

\address{$\text{}^1$Department of Mathematics, IIT Guwahati}

\email{p.parveenkumar144@gmail.com, b.bikash@iitg.ac.in}

\begin{abstract}
Let \( G \) be a finite non-cyclic group. Define \( \mathrm{Cyc}(G) \) as the set of all elements \( a \in G \) such that for any $b\in G$, the subgroup \( \langle a, b \rangle \) is cyclic. The \emph{non-cyclic graph} $\c$ of \( G \) is a simple undirected graph with vertex set \( G \setminus \mathrm{Cyc}(G) \), where two distinct vertices \( x \) and \( y \) are adjacent if the subgroup \( \langle x, y \rangle \) is not cyclic. An independent subset $C$ of the vertex set of a graph $\Gamma$ is called a perfect code of $\Gamma$ if every vertex of $V(\Gamma)\setminus C$ is adjacent to exactly one vertex in $C$. A subset \( T \) of the vertex set a graph \( \Gamma \) is said to be a \emph{total perfect code} if every vertex of \( \Gamma \) is adjacent to exactly one vertex in \( T \). In this paper, we prove that the graph $\c$ is Hamiltonian for any finite non-cyclic nilpotent group $G$. Also, we characterize all finite groups such that their non-cyclic graphs admit a perfect code. Finally, we prove that for a non-cyclic nilpotent group $G$, the non-cyclic graph $\c$ does not admit total perfect code.

 \end{abstract}


\keywords{Non-cyclic graph, Hamiltonian cycle, maximal cyclic subgroup, nilpotent groups \\ *  Corresponding author}

\maketitle
\section{Introduction}
Graphs associated with groups and other algebraic structures have been actively investigated due to their valuable applications in various fields, such as automata theory, data mining and network design (cf. \cite{a.AbawajyRees2016,kelarev2002ring,kelarev2003graph,a.kelarev2009cayley,kelarev2004labelled}). All the groups considered in this paper are finite. Let $G$ be a non-cyclic group and let \( x, y \in G \). The subgroup generated by \( x \) and \( y \), denoted \( \langle x, y \rangle \), is the smallest subgroup of \( G \) that contains both \( x \) and \( y \).  Note that
$
\langle x, y \rangle = \{ g_1 g_2 \dots g_n \mid g_i \in \{x, y \}, \, n \geq 1 \}.
$
 For $x\in G$, the \emph{cyclizer $\mathrm{Cyc}_G(x)$ of $x$}  is the set $\{y\in G : \langle x, y\rangle \text{ is cyclic\}}$. The \emph{cyclizer} $\cy$ of $G$ is the set $$\{x\in G: \langle x,y \rangle \text{ is cyclic for all } y\in G  \}.$$
Observe that $\cy$ is a normal subgroup of $G$. The \emph{non-cyclic graph} $\c$ of $G$ is a simple undirected graph whose vertex set is $G\setminus \cy$, and two vertices $x$ and $y$ are adjacent if the subgroup generated by $x$ and $y$ is non-cyclic.
In 2007, Abdollahi and Hassanabadi \cite{a.Abdollahi2007} introduced the concept of the non-cyclic graph and established basic graph theoretical properties.  In particular, they explored the regularity, clique number, and uniqueness properties of $\c$ in relation to the underlying group structure. In \cite{a.Abdollahi2009}, the authors characterized all finite groups whose non-cyclic graph has clique number at most $4$. Later, the concept of non-cyclic graphs of finite groups has been explored by several researchers. Costa et al. \cite{a.costaEulerian2018} studied the Eulerian properties of non-cyclic graphs of finite groups, demonstrating that certain conditions on the group's structure can guarantee the existence of Eulerian circuits. Ma and Su \cite{a.Ma2020Finite} classified finite non-cyclic groups whose non-cyclic graphs have positive genus. More recently, Ma et al. \cite{a.Maautomorphism2022} characterized the automorphism group of the non-cyclic graph of a finite group. 

In \cite{a.Abdollahi2009}, the authors obtained sufficient conditions for the non-cyclic graph $\c$ to be Hamiltonian.
\begin{lemma}{\rm \cite[Lemma 4.1]{a.Abdollahi2009}}{\label{Abdollahi Hamiltonian Lemma}}
    Let $G$ be a finite group such that $\Gamma(G/\cy)$ is Hamiltonian. Then $\Gamma (G)$ is also Hamiltonian.
\end{lemma}
\begin{proposition}{\rm \cite[Proposition 4.2]{a.Abdollahi2009}}{\label{Abdollahi Hamiltonian propostion}}
    Let $G$ be a finite non-cyclic group such that $$|G|+|\cy|>2|\mathrm{Cyc}_G(x)|$$ for all $x\in Z(G)\setminus \cy$. Then $\Gamma (G)$ is Hamiltonian. In particular, $\c$ is Hamiltonian whenever $Z(G)=\cy$.
\end{proposition}
 The condition given in Proposition \ref{Abdollahi Hamiltonian propostion} is not necessary. For example, consider the non-cyclic graph of the abelian group $\mathbb{Z}_4\times \mathbb{Z}_2$.
In this group, the element $(2,0)$ does not satisfy the condition given in Proposition \ref{Abdollahi Hamiltonian propostion}. However, the graph $\Gamma(\mathbb{Z}_4\times \mathbb{Z}_2)$ is Hamiltonian (see Figure \ref{Z2Z4}).

  Graph theory has been an essential area of mathematical research with numerous applications in coding theory, network design, and information security. One significant concept in this field is that of \textit{perfect codes}, which have been widely studied since the advent of coding theory in the late 1940s.  A \emph{code} in $\Gamma$ is a subset of $V(\Gamma)$. A code $C$ of $\Gamma$ is called a \emph{perfect code} if $C$ is an independent set such
that every vertex in $V(\Gamma)\setminus C$ is adjacent to exactly one vertex in $C$. 
The study of perfect codes has deep connections with information theory and combinatorial optimization, with significant contributions over the years (see \cite{Heden2008,Kratochvil1986}). It has been established that determining whether a given graph admits a perfect code is an NP-complete problem \cite{a.cull1999}. A code \( T \) of a graph \( \Gamma \) is said to be a \emph{total perfect code} if every vertex of \( \Gamma \) is adjacent to exactly one vertex in \( T \). If a graph admits a total perfect code, then by the definition, the subgraph induced by the total perfect code forms a matching.  It is known that determining whether a graph admits a total perfect code is NP-complete \cite{a.Gavlas1994}.
This computational complexity highlights the theoretical challenges in characterizing and classifying such codes in general graphs. The investigation of perfect codes has also extended to various graphs defined on groups, including Cayley graphs, power graphs, intersection power graphs, enhanced power graphs and Cayley sum graphs (see \cite{ a.maperfectproper,a.ma2017perfect,  a.Macayleysum, Martinez2009, Zhou2016}).

In this paper, we prove that the non-cyclic graph is Hamiltonian for all finite non-cyclic nilpotent groups. Also, we characterize all finite group $G$ such that the non-cyclic graph $\c$ admits a perfect code. Finally, we prove that for a finite non-cyclic nilpotent group $G$, the non-cyclic graph $\c$ does not admit a total perfect code.
\section{Preliminaries}

 A \emph{graph} $\Gamma$ is an ordered pair $(V(\Gamma), E(\Gamma))$, where $V(\Gamma)$ is a non-empty set of elements, called \emph{vertices}, and $E(\Gamma)$ is a set of unordered pairs of distinct elements of $V(\Gamma)$, called \emph{edges}. Two vertices $u$ and $v$ in a graph $\Gamma$ are said to be \emph{adjacent}, denoted $u\sim v$, if there exists an edge $\{u,v\} \in E(\Gamma)$. Otherwise, we write $u\nsim v$.  
 A vertex $v$ of a graph $\Gamma$ is called a \emph{dominating vertex} if $v$ is adjacent to all the vertices of $\Gamma$. If $\Gamma$ has a dominating vertex, then $\Gamma$ is called a \emph{dominatable graph}. A spanning path of $\Gamma$ is called a \emph{Hamiltonian path}. A spanning cycle of $\Gamma$ is called a \emph{Hamiltonian cycle}.  A graph containing a Hamiltonian cycle is called a \emph{Hamiltonian graph}.

Let $G$ be a finite group with identity element $e$. The order of an element $x\in G$ is denoted by $o(x)$. 
 Define the sets \( \pi(G) = \{ o(g) : g \neq e \in G \} \) and $\Omega_m(G)=\{g\in G : o(g)=m\}$. The \emph{exponent} $\mathrm{exp}(G)$ of \( G \) is the smallest positive integer \( n \) such that
$g^n = e$, for all $g \in G$. Indeed, $\mathrm{exp}(G)$  is the least common multiple of the orders of all the elements of $G$.  Define a relation \( \mathcal{R} \) on $G$ such that for \( g, h \in G \), we have \( g \mathcal{R} h \) if and only if \( \langle g \rangle = \langle h \rangle \). Clearly, the relation \( \mathcal{R} \) is an equivalence relation.  Let \( \mf \) denote the number of equivalence classes that contain elements of order \( d \). Observe that \( \mf \) is equal to the number of cyclic subgroups of order \( d \) in \( G \).
A \emph{maximal cyclic subgroup} of a group \( G \) is a cyclic subgroup that is not properly contained in any other cyclic subgroup of \( G \). The set of all maximal cyclic subgroups of \( G \) is denoted by \( \m \). Note that $|\m |=1$ if and only if $G$ is cyclic. 
Denote by $\mathbb{Z}_n$, the cyclic subgroup of order $n$. For $t \geq 1$, the dihedral $2$-group $\mathcal{D}_{2^{t+1}}$ is defined by the presentation   \[
        \mathcal{D}_{2^{t+1}}: = \langle a, b : a^{2^t} = 1, b^2 = 1, b^{-1} a b = a^{-1} \rangle.
        \]
        Note that \( \mathfrak{m}_2(\mathcal{D}_{2^{t+1}}) = 1 + 2^t \), and \( \mathfrak{m}_{2^j}(\mathcal{D}_{2^{t+1}}) = 1 \text{ for } 2 \leq j \leq t \).\\
 For $t \geq 2$, the generalized quaternion $2$-group $\mathcal{Q}_{2^{t+1}}$ is defined by the presentation
   \[
        \mathcal{Q}_{2^{t+1}} := \langle a, b : a^{2^t} = 1, a^{2^{t-1}} = b^2, b^{-1} a b = a^{-1} \rangle.
        \]
Observe that  \( \mathfrak{m}_4(\mathcal{Q}_{2^{t+1}}) = 1 + 2^{t-1} \), and \( \mathfrak{m}_{2^j}(\mathcal{Q}_{2^{t+1}}) = 1 \text{ for } j=1 \text{ and } 3 \leq j \leq t \).\\
 For $t \geq 3$, the semi-dihedral \( 2 \)-group $\mathcal{SD}_{2^{t+1}}$ is defined by the presentation
        \[
        \mathcal{SD}_{2^{t+1}} := \langle a, b : a^{2^t} = 1, b^2 = 1, b^{-1} a b = a^{-1+2^{t-1}} \rangle.
        \]
Note that \( \mathfrak{m}_2(\mathcal{SD}_{2^{t+1}}) = 1 + 2^{t-1} \), \( \mathfrak{m}_4(\mathcal{SD}_{2^{t+1}}) = 1 + 2^{t-2} \), and \( \mathfrak{m}_{2^j}(\mathcal{SD}_{2^{t+1}}) = 1 \text{ for } 3 \leq j \leq t \).

Let  \( H \) and \( K \) be subgroups of \( G \). The subgroup \( [H, K] \) of \( G \) is defined as the subgroup generated by all elements of the form \(  h^{-1} k^{-1} h k \), where \( h \in H \) and \( k \in K \). The lower central series of \( G \) is the descending chain of subgroups given by
\[
G \triangleright G^{(2)} \triangleright G^{(3)} \triangleright \cdots \triangleright G^{(i)} \triangleright G^{(i+1)} \triangleright \cdots,
\]
where \( G^{(2)} = [G, G] \) and \( G^{(i+1)} = [G^{(i)}, G] \) for \( i \geq 2 \). If the lower central series of $G$ terminates after finitely many non-trivial subgroups, the group \( G \) is called \emph{nilpotent}. All finite \( p \)-groups are examples of nilpotent groups.

\begin{lemma}{\rm \cite[Section 4, I]{a.Ma2020Finite}}{\label{p subgroups}}
    Let $p$ be a prime dividing the order of a
group $G$. Then $\mathfrak{m}_p(G) \equiv 1 (\mathrm{mod} \  p)$.
\end{lemma}

\begin{lemma}{\rm \cite[Theorem 5.4.10 (ii)]{a.Ma2020Finite}}{\label{unique p subgroups}}  Let $G$ be a $p$-group. Then $\mathfrak{m}_p(G)= 1$ if and only if either $G$ is a cyclic group or $G$ is a generalized quaternion group $2$-group.
\end{lemma}

A finite \( p \)-group of order \( p^n \) is said to have \emph{maximal class} if \( G^{(n-1)} \neq \{e\} \) and \( G^{(n)} = \{e\} \). For such groups, it holds that \( G / G^{(2)} \cong \mathbb{Z}_p \times \mathbb{Z}_p \) and \( G^{(i)} / G^{(i+1)} \cong \mathbb{Z}_p \) for \( 2 \leq i \leq n-1 \). Note that if $G$ is a cyclic $p$-group of order $p^t$, then $\mathfrak{m}_{p^i}(G)=1$ for $1\leq i\leq t$. The following theorem provides more insight into the class numbers of a finite \( p \)-group \( G \).

\begin{theorem}[{\cite{a.pgroupberkovi,b.pgroupisaac2006,a.pgroupkulakoff,a.pgroupmiller}}]{\label{Subgroups theorem}}
Let \( G \) be a finite non-cyclic \( p \)-group with exponent \( p^t \). If $G$ is not a $2$-group of maximal class, then
\begin{enumerate}
    \item[(i)] \( \mathfrak{m}_p(G) \equiv 1 + p \pmod{p^2} \), and
    \item[(ii)] \( p \mid \mathfrak{m}_{p^i}(G) \text{ for } 2 \leq i \leq t \).
\end{enumerate}
\end{theorem}

\begin{corollary}[{\cite{a.mishra2021lambda}}]{\label{subgroup corollary}}
Let \( G \) be a finite non-cyclic \( p \)-group of exponent \( p^t \). Then \( \mathfrak{m}_{p^i}(G) = 1 \) for some $i$ satisfying \( 1 \leq i \leq t \) if and only if \( G \) is isomorphic to 
   $ \mathcal{D}_{2^{t+1}}$,  $\mathcal{Q}_{2^{t+1}}$ or $ \mathcal{SD}_{2^{t+1}}$.
\end{corollary}

\begin{theorem}{\rm \cite{b.dummit1991abstract}}{\label{nilpotent}}
 Let $G$ be a non-trivial finite group. Then the following statements are equivalent:
 \begin{enumerate}
     \item[(i)] $G$ is a nilpotent group.
     \item[(ii)] Every Sylow subgroup of $G$ is normal.
    \item[(iii)] $G$ is the direct product of its Sylow subgroups.
    \item[(iv)] For $x,y\in G$, the elements $x$ and $y$ commute whenever $o(x)$ and $o(y)$ are relatively prime.
 \end{enumerate}
 \end{theorem}
 \begin{lemma}{\rm \cite[Lemma 2.11]{a.chattopadhyay2021minimal}}{\label{nilpotent maximal}}
Let $G$ be a finite nilpotent group such that $G = P_1P_2\cdots P_r$, where $P_1, P_2,\ldots ,P_r$ are the Sylow subgroups of $G$. Then any maximal cyclic subgroup of $G$  is of the form $M_1M_2 \cdots M_r$, where $M_i$ is a maximal cyclic subgroup of $P_i$, for $1 \leq i \leq r$.
\end{lemma}
Throughout the paper, we consider \( \mathcal{G} \) to be a nilpotent group of odd order with the identity element $e'$ and having no cyclic Sylow subgroups. We also consider  $\mathbb{P}$ to be a $2$-group with identity element $e''$ such that $\mathbb{P}$ is neither cyclic nor of maximal class. A non-cyclic nilpotent group \( G \) can be classified into one of the following categories:  

\begin{enumerate}
    \item \( G \cong \mathbb{Z}_n \times \mathcal{G} \), where \( \gcd(n, |\mathcal{G}|) = 1 \).  
    
    \item \( G \cong \mathbb{Z}_n \times \mathbb{P} \times \mathcal{G} \), where  \( \gcd(n, |\mathcal{G}|) = \gcd(2, n) = 1 \).  
    
    \item \( G \cong \mathbb{Z}_n \times \mathcal{Q}_{2^{t+1}} \times \mathcal{G} \), where  \( \gcd(n, |\mathcal{G}|) = \gcd(2, n) = 1 \).  
    
    \item \( G \cong \mathbb{Z}_n \times \mathcal{D}_{2^{t+1}} \times \mathcal{G} \), where  \( \gcd(n, |\mathcal{G}|) = \gcd(2, n) = 1 \).  
    
    \item \( G \cong \mathbb{Z}_n \times \mathcal{SD}_{2^{t+1}} \times \mathcal{G} \), where \( \gcd(n, |\mathcal{G}|) = \gcd(2, n) = 1 \).  
\end{enumerate}
For a positive integer \( n \), let \( \phi(n) \) represent the Euler's totient function.

\section{Hamiltonicity of $\c$}
In this section, we determine \(\mathrm{Cyc}(G)\) for finite non-cyclic nilpotent groups. We also show that the non-cyclic graph \(\Gamma(G)\) of a non-cyclic nilpotent group $G$ always has a Hamiltonian cycle.

\begin{lemma}{\label{Cyc lemma}}
   Let $G$ be a finite group such that  $G= G_1\times G_2 \times G_3$, where the orders $|G_1|,|G_2|$ and $|G_3|$ are pairwise relatively prime.   If $ x_i \in \langle y_i \rangle$ in $G_i$ for each $i\in \{1,2,3\}$, then $(x_1, x_2, x_3) \in \langle (y_1,y_2,y_3)\rangle$. 
\end{lemma}
\begin{proof}
    Let $|G_i|=n_i$ for each $i\in \{1,2,3\}$ and let $y= (y_1,y_2,y_3)$. Since $x_i\in \langle y_i \rangle$, there exists $m_i\in \mathbb{N}$ such that $x_i=  y_i^{m_i}$. Now $y^{m_1n_2n_3}=(x_1^{n_2n_3},e_2,e_3)$, where $e_i$ is the identity elements of $G_i$. Since $n_1, n_2, n_3$ are pairwise relatively prime, we obtain $\langle (x_1^{n_2n_3},e_2,e_3) \rangle = \langle (x_1,e_2,e_3)\rangle$. It follows that $(x_1,e_2,e_3) \in \langle y \rangle$. Similarly, $(e_1,x_2,e_3) \in \langle y \rangle$ and $(e_1,e_2,x_3) \in \langle y \rangle$. Consequently, $(x_1, x_2, x_3) \in \langle (y_1,y_2,y_3)\rangle$. This completes the proof.
\end{proof}
 Consider the subsets $S_1$ and $S_2$ of $\mathbb{Z}_{n} \times \mathcal{Q}_{2^{t+1}} \times \mathcal{G} $ given by \[
S_{1} = \left\{\left(\overline{x}, 1, e'\right) : \overline{x} \in \mathbb{Z}_{n}\right\} \text{ and } 
S_{2} = \left\{\left(\overline{x}, b^2, e'\right) :  \overline{x} \in \mathbb{Z}_{n}  \right\}.
\]
\begin{theorem}{\label{Cyclizer nilpotent}}
Let \( G \) be a finite non-cyclic nilpotent group.
Then
\[
\mathrm{Cyc}(G) = 
\begin{cases}
    \left\{\left(\overline{x}, e'\right) : \overline{x} \in \mathbb{Z}_{n}\right\} & \text{if } G = \mathbb{Z}_{n} \times \mathcal{G} \text{ and } \gcd\left(|\mathcal{G}|, n\right) = 1, \\[6pt]
    S_{1} \cup S_{2} & \text{if } G = \mathbb{Z}_{n} \times \mathcal{Q}_{2^{t+1}} \times \mathcal{G} \text{ and } \gcd\left(|\mathcal{G}|, n\right) = \gcd(n, 2) = 1, \\[6pt]
    \left\{\left(\overline{x}, e'', e'\right) : \overline{x} \in \mathbb{Z}_{n}\right\} & \text{if } G = \mathbb{Z}_{n} \times \mathbb{P} \times \mathcal{G} \text{ and } \gcd\left(|\mathcal{G}|, n\right) = \gcd(n, 2) = 1, \\[6pt]
    \left\{\left(\overline{x}, 1, e'\right) : \overline{x} \in \mathbb{Z}_{n}\right\} & \text{if } G = \mathbb{Z}_{n} \times \mathcal{D}_{2^{t+1}} \times \mathcal{G} \text{ and } \gcd\left(|\mathcal{G}|, n\right) = \gcd(n, 2) = 1, \\[6pt]
    \left\{\left(\overline{x}, 1, e'\right) : \overline{x} \in \mathbb{Z}_{n}\right\} & \text{if } G = \mathbb{Z}_{n} \times \mathcal{SD}_{2^{t+1}} \times \mathcal{G} \text{ and } \gcd\left(|\mathcal{G}|, n\right) = \gcd(n, 2) = 1.
\end{cases}
\]

\end{theorem}

\begin{proof}
    We consider the following three cases.\\
    \textbf{Case I:} $G= \mathbb{Z}_n \times \mathcal{G}$, where $\gcd\left(|\mathcal{G}|, n\right) = 1$. 
    Let $(\overline{x}_1,e')\in \{\left(\overline{x}, e'\right) : \overline{x} \in \mathbb{Z}_{n} \}$ and let $(\overline{x}_2,y)$ be an arbitrary element of $G$. As in the proof of Lemma \ref{Cyc lemma}, we obtain $(\overline{x}_1, e'), (\overline{x}_2, y)\in \langle (\overline{1}, y)\rangle$. It follows that $\langle (\overline{x}_1, e'), (\overline{x}_2, y) \rangle$ is cyclic. Thus $\{\left(\overline{x}, e'\right) : \overline{x} \in \mathbb{Z}_{n} \} \subseteq \mathrm{Cyc}(G)$.

    Suppose $(\overline{x}, z)\in \mathrm{Cyc}(G) \setminus \{\left(\overline{x}, e'\right) : \overline{x} \in \mathbb{Z}_{n} \}$. Then $z\neq e'$. Let $p$ be a prime such that $p \ \vert \ o(z)$. Then there exists an element $z'\in \langle z\rangle$ such that $o(z')=p$. Note that $\mathcal{G}$ has a Sylow $p$-subgroup. Applying Theorem \ref{Subgroups theorem} to that Sylow $p$-subgroup, we find at least two cyclic subgroups of $\mathcal{G}$ of order $p$. Thus, there exists an element $z''\in \mathcal{G}$ such that $o(z'')=p$ and $\langle z' \rangle \neq \langle z'' \rangle$. Note that both $(\overline{0}, z')$ and $(\overline{0}, z'')$ belong to the subgroup $\langle (\overline{x}, z), (\overline{0}, z'')\rangle$. It follows that $\langle (\overline{x}, z), (\overline{0}, z'')\rangle$ contains two cyclic subgroups of order $p$. Thus the subgroup $\langle (\overline{x}, z), (\overline{0}, z'')\rangle$ is non-cyclic, which is a contradiction to the assumption that $(\overline{x}, z)\in \mathrm{Cyc}(G)$. Thus $\{\left(\overline{x}, e'\right) : \overline{x} \in \mathbb{Z}_{n} \}= \mathrm{Cyc}(G)$.\\
    \textbf{Case II:} \( G = \mathbb{Z}_n \times \mathcal{Q}_{2^{t+1}} \times \mathcal{G} \), where \( \gcd(|\mathcal{G}|, n) = \gcd(n, 2) = 1 \).
    In this case, we show that $\mathrm{Cyc}(G)= S_1\cup S_2$, where $S_{1} = \left\{\left(\overline{x}, 1, e'\right) : \overline{x} \in \mathbb{Z}_{n}\right\}$  and $
S_{2} = \left\{\left(\overline{x}, b^2, e'\right) : \overline{x} \in \mathbb{Z}_{n}\right\}$.

    First suppose $(\overline{x}_1, 1, e')\in S_1$ and let $(\overline{x},y,z)$ be an arbitrary element of $G$. By Lemma \ref{Cyc lemma}, both the elements $(\overline{x},y,z)$ and $(\overline{x}_1, 1, e')$ are contained in the cyclic subgroup generated by $(\overline{1},y,z)$. It follows that $\langle (\overline{x},y,z),(\overline{x}_1, 1, e')\rangle$ is cyclic. Consequently, $S_1\subseteq \mathrm{Cyc}(G)$.
    
    Now suppose $(\overline{x}_1, b^2, e') \in S_2$ and $(\overline{x},y,z)\in G$. Observe that these two elements belong to the cyclic subgroup generated by $(\overline{1},y',z)$, where $y' \in \mathcal{Q}_{2^{t+1}}$ such that $ y,b^2\in \langle y'\rangle$. Such an $y'$ exists because  $b^2$ is contained in every non-trivial cyclic subgroup of $\mathcal{Q}_{2^{t+1}}$. It follows that the subgroup $\langle (\overline{x}_1, b^2, e') , (\overline{x},y,z)\rangle$ is cyclic. Therefore, $S_2\subseteq \mathrm{Cyc}(G)$.

    Conversely, suppose $(\overline{x},y,z)\in \cy \setminus (S_1\cup S_2)$. It follows that either $y\notin \{1, b^2\}$ or $z\neq e'$. First assume that $y\notin \{1, b^2\}$. Therefore $o(y)> 2$. Let $g\in \langle y\rangle$ be such that $o(g)=4$. Note that $\mathfrak{m}_4({\mathcal{Q}_{2^{t+1}}})\geq 2$. Therefore there exists $g'\in \mathcal{Q}_{2^{t+1}}$ such that $o(g')=4$ and $\langle g \rangle \neq \langle g' \rangle$. Thus $\langle (\overline{0}, g, e')\rangle$ and $\langle (\overline{0}, g', e')\rangle$ are distinct cyclic subgroups of order $4$ in $G$. Also, note that  \linebreak[4] $(\overline{0}, g, e'), (\overline{0}, g', e') \in \langle (\overline{x},y,z), (\overline{0}, g', e')\rangle$. Therefore, the subgroup $ \langle (\overline{x},y,z), (\overline{0}, g', e')\rangle$ is non-cyclic, a contradiction to the fact that $(\overline{x},y,z)\in \cy$. 
    
    Now suppose $z\neq e'$.  Then there exists a prime $p$ such that $p\ \vert \  o(z)$. Consider $z'\in \langle z \rangle$ such that $o(z')=p$. Now there exists $z''\in \mathcal{G}$ such that $o(z'') = p$ and $\langle z' \rangle \neq \langle z'' \rangle$. Then $\langle (\overline{0}, 1, z') \rangle$ and $\langle (\overline{0}, 1, z'') \rangle$ are two distinct cyclic subgroups of  $\langle (\overline{x},y,z),(\overline{0}, 1, z'') \rangle$. It follows that $(\overline{x},y,z)\notin \cy$; a contradiction. Thus $\cy =S_1\cup S_2$.\\
    \textbf{Case III:} \( G = \mathbb{Z}_n \times P \times \mathcal{G} \), where \( \gcd(|\mathcal{G}|, n) = \gcd(n, 2) = 1 \), and \( P \in \{\mathbb{P}, \mathcal{D}_{2^{t+1}}, \mathcal{SD}_{2^{t+1}} \} \). Let $e_1$ be the identity element of $P$.  Suppose $(\overline{x}_1, e_1, e')\in \{(\overline{x}, e_1, e') : \overline{x} \in \mathbb{Z}_n\}$ and $(\overline{x}_2,y,z)\in G$.  As earlier, 
    $\langle (\overline{x}_1, e_1, e'), (\overline{x}_2,y,z)\rangle$ is a cyclic subgroup of $G$. Consequently, $\{(\overline{x}, e_1, e') : \overline{x} \in \mathbb{Z}_n\}\subseteq \cy$.

    Conversely, suppose $(\overline{x}, y ,z) \in \cy \setminus \{(\overline{x}, e_1, e') : \overline{x} \in \mathbb{Z}_n\}$. Then either $y\neq e_1$ or $z\neq e'$. First assume that $y\neq e_1$. Then there exists an element $y'\in \langle y \rangle$ such that $o(y')=2$. Consider an element $y''\in P$ such that $o(y'')=2$ and $\langle y' \rangle \neq \langle y'' \rangle$. Then $\langle (\overline{0}, y', e')\rangle$ and $\langle (\overline{0}, y'', e')\rangle$ are two distinct subgroups of $\langle (\overline{x}, y ,z) , (\overline{0}, y'', e')\rangle$, contradicting the assumption that $(\overline{x}, y ,z) \in \cy$. Similarly, we get a contradiction for the case $z\neq e'$ as well. Thus $\cy =\{(\overline{x}, e_1, e') : \overline{x} \in \mathbb{Z}_n\}$.
\end{proof}

\begin{lemma}{\label{multipartite Lemma}}
     Let $m \in \pi (G)$ such that $\mathfrak{m}_m(G)=s\geq 2$. Then the subgraph of $\c$ induced by  $\Omega_m(G)$ is isomorphic to the complete $s$-partite graph with exactly $\phi(m)$ vertices in each partition set. 
\end{lemma}
\begin{proof}
    Let $x\in \Omega_m(G)$ be an arbitrary element. Let $[x]$ be the equivalence class of $\mathcal{R}$ containing $x$. From the definition of $\mathcal{R}$, it is clear that  $[x]$ is an independent set in $\c$. 

    Now suppose $y\in \Omega_m(G)$ such that $\langle x\rangle \neq \langle y \rangle$. Consider $x_1\in [x]$ and $y_1\in [y]$. Then $\langle x_1, y_1\rangle$ is non-cyclic, and so $x_1\sim y_1$ in $\c$. Thus the result follows.
\end{proof}

\begin{theorem}{\label{pgroup Hamiltonian}}
    Let $G=\mathbb{Z}_n\times P$, where $\mathrm{gcd}(n, |P|)=1$ and $P$ is $p$-group such that $P$ is neither cyclic nor a $2$-group of maximal class. Then $\Gamma(G)$ is Hamiltonian.
\end{theorem}
\begin{proof}
    Let $e_1$ be the identity element of $P$. Clearly, $\cy = \mathbb{Z}_n\times \{e_1\}$. In view of Lemma \ref{Abdollahi Hamiltonian Lemma},  it is sufficient to show that $\Gamma (P)$ is Hamiltonian. Notice that $\mathrm{Cyc}(P)=\{e_1\}$. Let $\pi (G)=\{p, \ldots ,p^\alpha \}$. First, suppose that $\alpha =1$. Then $V(\Gamma(P))= \Omega _p(P)$. By Theorem \ref{Subgroups theorem} and Lemma \ref{multipartite Lemma}, $\Gamma(P)$ is a complete multipartite graph with exactly $p-1$ vertices in each partite set. Thus $\Gamma(P)$ is Hamiltonian. Now assume that $\alpha \geq 2$. 
    For $j\in\{1, \ldots ,\alpha\}$, let $\Gamma_{\alpha-j+1}$ be the subgraph of $\Gamma(P)$ induced by $\Omega_{p^{j}}(P)$. Note that $\Gamma_{\alpha-j+1}$ is a complete multipartite graph with the same number of vertices in each partite set. Therefore $\Gamma_{\alpha-j+1}$ is Hamiltonian for each $j$.

    Let $H_{1}$ be a Hamiltonian path in $\Gamma_{1}$ with initial and terminal vertices  $x_1$ and $y_1$, respectively. Theorem \ref{Subgroups theorem} guarantees an  $x_{2}\in \Omega_{p^{\alpha-1}}(P)$ such that $x_{2}\notin \langle y_1 \rangle$.  Now choose a Hamiltonian path $H_{2}$ in $\Gamma_{2}$ with initial vertex $x_{2}$. Let the terminal vertex of $H_{2}$ be $y_{2}$. Observe that $y_{1}\sim x_{2}$.  In general, for $1\leq j\leq \alpha-1$, suppose the Hamiltonian paths, $H_1,\ldots , H_j$ with initial vertices $x_1, \ldots , x_j$ and terminal vertices $y_1, \ldots ,y_j$, are obtained in the subgraphs $\Gamma_1, \ldots , \Gamma_j$, respectively, such that $x_i\notin \langle y_{i-1}\rangle$ for $2\leq i \leq j$. Now we choose $x_{j+1}\in \Omega_{p^{j+1}}(P)$ such that $x_{j+1}\notin \langle y_{j}\rangle$, and an Hamilotian path $H_{j+1}$ with initial vertex $x_{j+1}$ in $\Gamma_{j+1}$. Let the terminal vertex of $H_{j+1}$ be $y_{j+1}$. Observe that $y_j\sim x_{j+1}$ for $2\leq j \leq \alpha-1$. Therefore the paths $H_1, \ldots , H_\alpha$ altogether produce a Hamiltonian path $H$ in $\Gamma(P)$ with initial and terminal vertices $x_{1}$ and $y_\alpha$, respectively. Note that $\mathfrak{m}_p(P)\geq 3$. Therefore we can choose $y_\alpha$ in such a way that $y_\alpha \notin \langle x_1\rangle$, giving $y_\alpha\sim x_1$. Thus the Hamiltonian path $H$ along with the edge $\{x_1, y_\alpha \}$ produce a Hamiltonian cycle in $\Gamma(P)$. This completes the proof.
\end{proof}
The following example illustrates Theorem \ref{pgroup Hamiltonian}. It demonstrates the Hamiltonian cycle in the non-cyclic graph of $\mathbb{Z}_4 \times \mathbb{Z}_2$. Note that $\Omega_{4}(\mathbb{Z}_4 \times \mathbb{Z}_2)=\{(1,0), (3,0), (1,1), (3,1)\}$ and \linebreak[4] $\Omega_{2}(\mathbb{Z}_4 \times \mathbb{Z}_2)=\{(2,0), (0,1), (2,1)\}$.
 \begin{figure}[ht]
    \centering
    \includegraphics[width=0.45\textwidth]{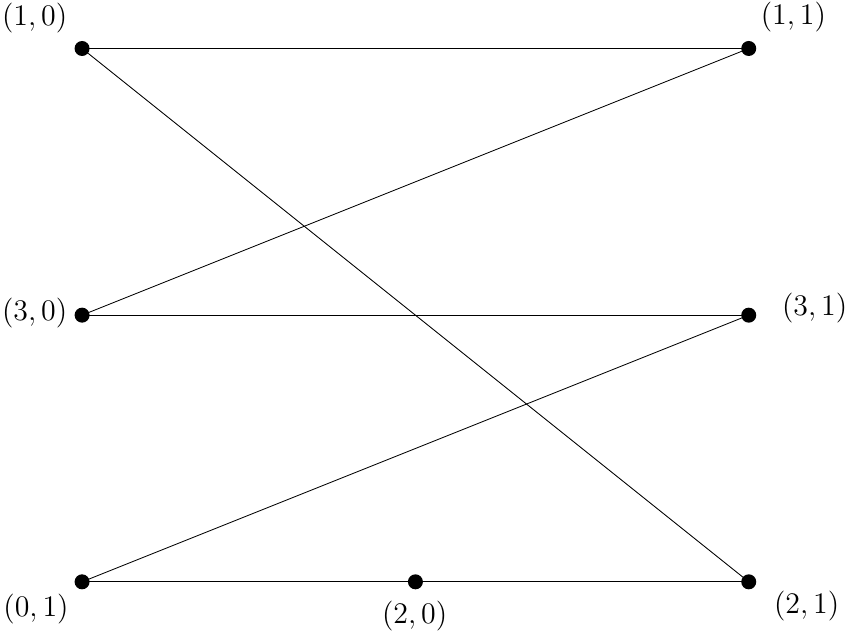}
    \caption{Hamiltonian cycle in $\Gamma(\mathbb{Z}_4 \times \mathbb{Z}_2)$.}
    \label{Z2Z4}
\end{figure}
\begin{lemma}{\label{m_d>2 G}}
Let $P$ be a $2$-group. Then $\mathfrak{m}_{o(g)}(P \times \mathcal{G})\geq 3$ for $g\in \{(y,z )\in P \times \mathcal{G}: z\neq e'\}$.
\end{lemma}
\begin{proof}
    Let $g=(y,z )\in P \times \mathcal{G}$ be such that $z\neq e'$. Since $\mathcal{G}$ is nilpotent, $\mathcal{G}$ can be written as a direct product of its Sylow subgroups. For $j\in\{1,\ldots , k \}$, suppose $P_j$ be the Sylow $p_j$-subgroup of $\mathcal{G}$ such that $\mathcal{G}=P_1\times \cdots \times P_k$. Now write $z$ as $(z_1, \ldots , z_k)$. Then there exists at least one  $i\in \{1, \ldots , k\}$ such that $o(z_i)=p_i^\ell$, where $\ell\geq 1$. By Theorem \ref{Subgroups theorem}, there exists $z_i', z_i''\in P_i$ such that $o(z_i)=o(z_i')=o(z_i'')$, and the cyclic subgroups $\langle z_i\rangle, \langle z_i'\rangle$ and  $\langle z_i''\rangle$ are pairwise distinct. Replacing the $i$-th entry $z_i$ in $z$ by $z_i'$ and $z_i''$, respectively we write $z'=(z_1, \ldots , z_i', \ldots , z_k)$ and $z''=(z_1,  \ldots , z_i'',  \ldots , z_k)$ in $\mathcal{G}$. Let $g'=( y,z' )$ and $g''= ( y,z'')$. Then $o(g)=o(g')=o(g'')$. Also, $\langle g\rangle$, $\langle g' \rangle$ and $\langle g''\rangle$ are three pairwise distinct cyclic subgroups of $P \times \mathcal{G}$ as they contain three distinct cyclic subgroups of $p_i^\ell$. Hence $\mathfrak{m}_{o(g)}(G)\geq 3$.   
\end{proof}
The proof of the following lemma is similar to the proof of Lemma \ref{m_d>2 G}. Hence the proof is omitted.
\begin{lemma}{\label{m_d>2 P G}}
    Let $G=  \mathbb{P} \times \mathcal{G}$. Then $\mathfrak{m}_{o(g)}(G)\geq 2$ for all $g\in G\setminus \cy$.
\end{lemma}

\begin{lemma}{\label{Hamiltonian path lemma}}
    Let $G=  \mathbb{P} \times \mathcal{G}$ and let $x, y \in \Omega _m(G)$ such that $\langle x \rangle \neq \langle y \rangle$. Then there exists a Hamiltonian path from $x$ to $y$ in the subgraph of $\c$ induced by $\Omega_m(G)$.
\end{lemma}
\begin{proof}
    The result directly follows from Lemma \ref{m_d>2 P G} and Lemma \ref{multipartite Lemma}.
\end{proof}
    
\begin{theorem}{\label{ZnPG Hamiltonian}}
    Let  $G= \mathbb{Z}_n  \times \mathbb{P} \times  \mathcal{G}$. Then $\c$ is Hamiltonian.
\end{theorem}
\begin{proof}

    In view of Lemma \ref{Abdollahi Hamiltonian Lemma} and Theorem \ref{pgroup Hamiltonian}, we prove that $\Gamma (\mathbb{P} \times \mathcal{G})$ is Hamiltonian, where $\mathcal{G}$ is non-trivial. Let $P_1,P_2, \ldots , P_k$ be all the Sylow $p_i$-subgroups of $G'$ such that $p_1>p_2>\cdots >p_k$, where $G'=\mathbb{P} \times \mathcal{G}$. Note that $G'=P_1\times P_2\times \cdots \times P_k$, where $P_k=\mathbb{P}$. Suppose $\mathrm{exp}(G')=p_1^{\alpha_1}p_2^{\alpha_2} \cdots p_k^{\alpha_k}$.  
   \textbf{Case 1:} First consider that $\mathrm{exp}(G')=p_1p_2$. Notice that  $\pi(G')=\{p_1,p_2,p_1p_2\}$. Applying Theorem \ref{Subgroups theorem} on the Sylow $p_i$-subgroups of $G'$, we get $\mathfrak{m}_{p_1}(G')\geq p_1+1$ and $\mathfrak{m}_{p_2}(G')\geq p_2+1$. It implies that $\mathfrak{m}_{p_1p_2}(G')=s \geq (p_1+1)(p_2+1)$. Let $\Gamma_1$ be the subgraph of $\Gamma(G')$ induced by $\Omega _{p_1}(G')$. Note that $\Gamma_1$ is a complete multipartite graph with $p_1-1$ elements in each partite set. Therefore $\Gamma_1$ is Hamiltonian. 

   Let $H_1$ be a Hamiltonian path in $\Gamma_1$ with initial and terminal vertices $x_1$ and $x_2$, respectively.  Consider $y_1, y_2\in \Omega_{p_1p_2}(G')$ such that $x_2\notin \langle y_1 \rangle$ and $\langle y_1\rangle \neq \langle y_2 \rangle $. It follows that $\langle x_2,y_1\rangle$ is non-cyclic and so $x_2\sim y_1$. Let $\Gamma_2$ be the subgraph of $\Gamma(G')$ induced by $(\langle y_1\rangle \cup \langle y_2\rangle)\cap \Omega_{p_1p_2}(G')$. Observe that $\Gamma_2$ is a complete bipartite graph with $(p_1-1)(p_2-1)$ elements in each partite set. Let $H_2$ be a Hamiltonian path in $\Gamma_2$ with initial and terminal vertices $y_1$ and $y_2$, respectively. Consider $z_1, z_2\in \Omega_{p_2}$ such that $z_1\notin \langle y_2\rangle$ and $\langle z_1\rangle \neq \langle z_2 \rangle$. Let $\Gamma_3$ be the subgraph of $\Gamma(G')$ induced by $\Omega _{p_2}(G')$. Note that $\Gamma_3$ is a complete multipartite graph with $p_2-1$ elements in each partite set. Let $H_3$ be a Hamiltonian path in $\Gamma_3$ with initial and terminal vertices $z_1$ and $z_2$, respectively. 
   Let $y_3, y_4, \ldots, y_s \in \Omega_{p_1p_2}(G')$ such that $\langle y_i\rangle \neq \langle y_j \rangle$ for distinct $i,j \in \{1,2,3,\ldots ,s\}$. We choose $y_3$ and $y_s$ in such a way that $z_2\notin \langle y_3 \rangle$ and $x_1\notin \langle y_s \rangle$. For example, one can choose $y_3$ to be an element of order $p_1p_2$ in $\langle x_1,z_1\rangle$.
    Let $\Gamma_4$ be the subgraph of $\Gamma (G')$ induced by the set $(\bigcup\limits_{i=3}^s \langle y_i\rangle) \cap \Omega_{p_1p_2}(G')$. Observe that $\Gamma_4$ is a complete $(s-2)$-partite graph with exactly $(p_1-1)(p_2-1)$ elements in each partite. Thus, we have a Hamiltonian path $H_4$ from $y_3$ to $y_s$. Note that for each $i\in \{1,2,3\}$, the terminal vertex of $H_i$ is adjacent with the initial vertex of $H_{i+1}$. Also, $x_1\notin \langle y_s \rangle$ implies that $y_s \sim x_1$. Consequently, we get a Hamiltonian cycle in $\Gamma(G')$. Hence $\Gamma(G')$ is Hamiltonian.\\
    \textbf{Case 2:} Now consider $\mathrm{exp}(G')\neq p_1p_2$. 
    For $1\leq i \leq k-1$ and $1\leq j \leq \alpha_i$ define $$\mathcal{B}_{i,j}=\{p_i^jp_{i+1}^{\gamma _{i+1}}\cdots p_k^{\gamma_k} : 0\leq \gamma_r \leq \alpha_r \text{ for } i+1\leq r\leq k\} ,$$
     and for $1\leq j \leq \alpha_k$ define 
    $$\mathcal{B}_{k,j}=\{p_k^{j} \}.$$
    Then $\pi(G)= \mathcal{B}_{1,\alpha_{1}}\cup \mathcal{B}_{1,\alpha_1-1} \cup \cdots \cup \mathcal{B}_{1, 1}\cup \mathcal{B}_{2,\alpha_2}\cup \cdots \cup \mathcal{B}_{k,1}$. We arrange all the elements of these sets in ascending order except for $\mathcal{B}_{1,\alpha_{1}}$. In $\mathcal{B}_{1,\alpha_{1}}$, we list the element $p_1^{\alpha_1}p_k$ at the beginning, and the remaining elements in ascending order afterwards. Thus, $$\pi(G)=\{\underbrace {p_1^{\alpha_1}p_k, \  p_1^{\alpha_1},  \ldots , \  p_1^{\alpha_1}p_2^{\alpha_2}\cdots p_k^{\alpha_k}}_{\mathcal{B}_{1,\alpha_1}} \ , \ldots ,  \ \underbrace{p_k}_{\mathcal{B}_{k,1}}\}.$$ Before proceeding further, we first prove the following claims.\\
    \textbf{Claim I:} Let $x\in G'$ such that $o(x)\in \mathcal{B}_{i,j}$. Then for any $m\in \mathcal{B}_{i,j}$, there exists a $y\in \Omega _m(G)$ such that $\langle x, y\rangle$ is not cyclic. \\
    \textit{Proof of the claim.}
    Suppose $x=(x_1,\ldots, x_i,  \ldots , x_k)$ such that $o(x)\in \mathcal{B}_{i,j}$.
    
 First, assume that $i=1$. Observe that $o(x_1)=p_1^{j}$. By Theorem \ref{Subgroups theorem}, there exists $y_1\in P_1$ such that $o(y_1)=p_1^{j}$ and $\langle x_1\rangle \neq \langle y_1\rangle$. Let  $m= p_1^{j} p_{2}^{\delta_{2}}\cdots p_k^{\delta_k}\in \mathcal{B}_{1,j}$. For $2\leq r\leq  k$, consider $y_r\in P_r$ such that $o(y_r)=p_r^{\delta_r}$. Then the element $y=(y_1,y_2, \ldots   , y_k)$ belongs to the set $\Omega _m(G)$. Also, notice that $\langle x , y \rangle$ is a non-cyclic group as it contains two distinct cyclic subgroups of order $p_1^{j}$. 
 
 Now suppose $1<i<k$. In this case, note that $x_\ell=e'$ for $1\leq \ell\leq i-1$. Also, observe that $o(x_i)=p_i^{j}$. By Theorem \ref{Subgroups theorem}, there exists $y_i\in P_i$ such that $o(y_i)=p_i^{j}$ and $\langle x_i\rangle \neq \langle y_i \rangle$. Let $m= p_i^{j} p_{i+1}^{\delta_{i+1}}\cdots p_k^{\delta_k}\in \mathcal{B}_{i,j}$. For $i+1\leq r\leq  k$, consider $y_r\in P_r$ such that $o(y_r)=p_r^{\delta_r}$ . Then the element $y=(e', \ldots  ,e',y_i, \ldots , y_k)$ belongs to the set $\Omega _m(G)$. Also, $\langle x , y \rangle$ is a non-cyclic group. 

 Finally, assume that $i=k$. In this case, note that $x_\ell=e'$ for $1\leq \ell\leq k-1$ and $o(x_k)=p_k^{j}$. By Theorem \ref{Subgroups theorem}, there exists $y_k\in P_k$ such that $o(y_k)=p_k^{j}$ and $\langle x_k\rangle \neq \langle y_k \rangle$. Let $m= p_k^{j}$. Then the element $y=(e', \ldots  ,e', y_k)$ belongs to the set $\Omega _m(G)$. Also,  $\langle x , y \rangle$ is a non-cyclic group. This proves the claim.
 
For $1\leq i \leq k$ and $1\leq \beta \leq \alpha_i$, let $\Gamma_{i,\beta}$ be the subgraph of $\Gamma(G')$ induced by  $\{z\in G' : o(z)\in \mathcal{B}_{i,\beta}\}$. 
    \textbf{Claim II:} For $1\leq i \leq k-1$   there exists a Hamiltonian path $H_{i,\beta}$ in $\Gamma_{i,\beta}$ from $x\in \Omega_{p_i^\beta}(G')$ to some $y\in \Omega_{p_i^\beta p_{i+1}^{\alpha_{i+1}}\cdots p_k^{\alpha_k}}(G')$.\\
 \textit{Proof of the claim.} Let $\mathcal{B}_{i,\beta}=\{m_1,\ldots , m_\mu \}$ such that $p_i^\beta= m_1<\cdots <m_\mu=p_i^\beta p_{i+1}^{\alpha_{i+1}}\cdots p_k^{\alpha_k}$. For $j\in\{1, \ldots ,\mu\}$, let $\Gamma_{j}$ be the subgraph of $\Gamma_{i,\beta}$ induced by $\Omega_{m_{j}}(G')$. Note that $\Gamma_{j}$ is a complete multipartite graph with the same number of vertices in each partite set. Therefore $\Gamma_{j}$ is Hamiltonian for each $j$.
    Let $H_{1}$ be a Hamiltonian path in $\Gamma_{1}$ with initial and terminal vertices  $x_1\in \Omega_{p_i^\beta}(G')$ and $y_1$, respectively.   \textbf{Claim I} guarantees an  $x_{2}\in \Omega_{m_2}(G) $ such that $x_{2}\sim   y_1 $.  Now choose a Hamiltonian path $H_{2}$ in $\Gamma_{2}$ with initial vertex $x_{2}$. Let the terminal vertex of $H_{2}$ be $y_{2}$.  In general, for $1\leq j\leq \mu-1$, once a Hamiltonian path $H
    _{j}$ is chosen in $\Gamma_{j}$ with initial and terminal vertices $x_{j}$ and $y_{j}$, respectively, we choose a Hamiltonian path $H_{j+1}$ in $\Gamma_{j+1}$ with initial and terminal vertices $x_{j+1}$ and $y_{j+1}$, respectively, such that $y_{j}\sim x_{j+1}$. The paths $H_1, \ldots , H_\mu$ altogether produce a Hamiltonian path $H_{i,\beta}$ in $\Gamma_{i,\beta}$ with initial and terminal vertices $x_{1}$ and $y_\mu\in \Omega_{p_i^\beta p_{i+1}^{\alpha_{i+1}}\cdots p_k^{\alpha_k}}(G')$, respectively.  This completes the proof of the claim.
    \\
     \textbf{Claim III:}  Let $x\in G'$ such that $o(x)=p_i^j p_{i+1}^{\alpha_{i+1}}\cdots p_k^{\alpha_k}$, where $1\leq i\leq k-1$ and $j\geq 2$. Then there exists $y\in \Omega _{p_i^{j-1}}(G')$ such that $\langle x, y\rangle$ is not cyclic. \\
     \textit{Proof of the claim.}  Suppose $x=(x_1,\ldots , x_i,  \ldots , x_k)$, such that $o(x)=p_i^j p_{i+1}^{\alpha_{i+1}}\cdots p_k^{\alpha_k}$.  Observe that $o(x_i)=p_i^{j}$. Let $x_i'\in \langle x_i\rangle$ such that $o(x_i')=p_i^{j-1}$. By Theorem \ref{Subgroups theorem}, there exists $y_i'\in P_i$ such that $o(y_i')=p_i^{j-1}$ and $\langle x_i'\rangle \neq \langle y_i' \rangle$.  Let $y=(e', \ldots  ,e',y_i',e', \ldots , e', e'')$. Then the element $y$ belongs to the set $\Omega _{p_i^{j-1}}(G')$. Then $\langle x , y \rangle$ is a non-cyclic group. This proves the claim. \\
\textbf{Claim IV:} Let $x\in G'$ such that $o(x)=p_i p_{i+1}^{\alpha_{i+1}}\cdots p_k^{\alpha_k}$, where $1\leq i\leq k-1$. Then there exists  $y\in \Omega _{p_{i+1}^{\alpha_{i+1}}}(G')$ such that $\langle x, y\rangle$ is not cyclic. \\
  \textit{Proof of the claim.}    Suppose $x=(x_1,\ldots , x_i,  \ldots , x_k)$ such that $o(x)=p_i p_{i+1}^{\alpha_{i+1}}\cdots p_k^{\alpha_k}$. Then we have $o(x_{i+1})=p_{i+1}^{\alpha_{i+1}}$. By Theorem \ref{Subgroups theorem}, there exists  $y_{i+1}\in P_{i+1}$ such that $\langle x_{i+1}\rangle \neq \langle y_{i+1}\rangle$ and $o(y_{i+1})=p_{i+1}^{\alpha_{i+1}}$. Let $y=(e', \ldots  ,e',y_{i+1},e', \ldots , e', e'')$. Then $\langle x,y \rangle$ is a non-cyclic group. This proves the claim.
     
    Now, we discover a Hamiltonian path in $\Gamma (G')$.  Consider $x_{1,\alpha_1}\in \Omega_{p_1^{\alpha_1}p_k}(G')$. As in the proof of \textbf{Claim II}, we get a Hamiltonian path $H_{1,\alpha_1}$ with initial vertex $x_{1,\alpha_1}$ in $\Gamma_{1,\alpha _1}$. Let $y_{1,\alpha_1}$ be the terminal vertex of $H_{1,\alpha_1}$. Note that $y_{1,\alpha_1}\in \Omega_{p_1^{\alpha_1}p_2^{\alpha_2} \cdots p_k^{\alpha_k}}(G')$.  If $\alpha_1\geq 2$, then by \textbf{Claim III}, there exists an element $x_{1, \alpha_1-1} \in \Omega_{p_1^{\alpha_1-1}}(G')$ such that $y_{1,\alpha_1}\sim x_{1, \alpha_1-1}$. By \textbf{Claim II}, we get a Hamiltonian path $H_{1,\alpha_1-1}$ with initial vertex $x_{1,\alpha_1-1}$ in $\Gamma_{1,\alpha _1}$. Let $y_{1,\alpha_1-1}$ be the terminal vertex of $H_{1,\alpha_1-1}$. Note that $y_{1,\alpha_1-1}\in \Omega_{p_1^{\alpha_1-1}p_2^{\alpha_2} \cdots p_k^{\alpha_k}}(G')$. 
  In general, for $1\leq \ell \leq \alpha_1-1$, once a Hamiltonian path $H_{1,\alpha_1- \ell+1}$ is chosen in $\Gamma_{1,\alpha_1- \ell+1}$ with initial and terminal vertices $x_{1,\alpha_1- \ell+1}$ and $y_{1,\alpha_1- \ell+1}$, respectively, we choose a Hamiltonian path $H_{1,\alpha_1- \ell}$ in $\Gamma_{1, \alpha_1- \ell}$ with initial vertex $x_{1, \alpha_1- \ell}$ such that $y_{1, \alpha_1- \ell+1}\sim x_{1, \alpha_1- \ell}$. The paths $H_{1,\alpha_1}, H_{1,\alpha_1-1} \ldots , H_{1,1}$ altogether produce a Hamiltonian path $\mathcal{H}_1$ in the subgraph of $\Gamma(G') $ induced by the set of vertices whose order belongs to  $\mathcal{B}_{1,\alpha _1}\cup  \cdots \cup \mathcal{B}_{1,1}$ with initial and terminal vertices $x_{1,\alpha_1}$ and $y_{1,1}$, respectively.

    Note that $y_{1,1}\in \Omega_{p_1p_2^{\alpha_2}\ldots p_k^{\alpha_k}}$. By \textbf{Claim IV}, there exists an element $x_{2, \alpha_2} \in \Omega_{p_2^{\alpha_2}}(G')$ such that $y_{1,1}\sim x_{2, \alpha_2}$.  Proceeding as in the previous paragraph we get a Hamiltonian path $\mathcal{H}_2$ in the subgraph of $\Gamma(G') $ induced by the set of vertices whose order belongs to   $\mathcal{B}_{2,\alpha _2}\cup  \cdots \cup \mathcal{B}_{2,1}$, with initial and terminal vertices $x_{2,\alpha_2}$ and $y_{2,1}$, respectively.

    In general, for each $i\in\{1,\ldots , k-1\}$, we get a Hamiltonian path $\mathcal{H}_i$ in the subgraph induced by the set of vertices whose order belongs to   $\mathcal{B}_{i,\alpha _i}\cup  \cdots \cup \mathcal{B}_{i,1}$. Further, for $1\leq i\leq k-2$, the terminal vertex of $\mathcal{H}_i$ is adjacent to the initial vertex of $\mathcal{H}_{i+1}$. Let the terminal vertex of $\mathcal{H}_{k-1}$ be $y_{k-1,1}$. Note that $y_{k-1,1}\in \Omega_{p_{k-1}p_k^{\alpha_k}}$.   

 By \textbf{Claim IV}, there exists an element $x_{k,\alpha_k}\in \Omega_{p_k^{\alpha_k}}$ such that $y_{k-1,1}\sim x_{k,\alpha_k}$. Note that for $1\leq j \leq \alpha_k$, we have $\mathcal{B}_{k,j}=\{p_k^j\}$. Using Lemma \ref{Hamiltonian path lemma}, we get Hamiltonian paths  $H_{k,\alpha_k},  \ldots , H_{k,1}$ in $\Gamma_{k,\alpha_k},\ldots , \Gamma_{k,1} $, respectively. Further, the initial vertex of $H_{k,\alpha_k}$ is $x_{k,\alpha_k}$, and due to Theorem \ref{Subgroups theorem}, initial vertices of the remaining paths can be chosen in such a way that the terminal vertex of $H_{k,\alpha_k-\ell+1}$ is adjacent to the initial vertex of $H_{k,\alpha_k-\ell}$ for $1\leq \ell \leq \alpha_k-1$. Let the terminal vertex of $H_{k,1}$ be $y_{k,1}$. The paths $H_{k,\alpha_k}, H_{k,\alpha_k-1} \ldots , H_{k,1}$ altogether produce a Hamiltonian path $\mathcal{H}_k$ in the subgraph of $\Gamma(G') $ induced by the set of vertices whose order belongs to   $\mathcal{B}_{k,\alpha _k}\cup  \cdots \cup \mathcal{B}_{k,1}$ with initial and terminal vertices $x_{k,\alpha_k}$ and $y_{k,1}$, respectively. Clearly, $o(y_{k,1})=p_k$. Since $\mathfrak{m}_{p_k}(G')\geq 3$, one can choose $y_{k,1}$ in such a way that $y_{k,1}\notin \langle x_{1,\alpha_k}\rangle$. It follows that $y_{k,1}\sim x_{1,\alpha_k}$. 

 Finally, the paths $\mathcal{H}_1,\mathcal{H}_2,\ldots, \mathcal{H}_k$ altogether produce a Hamiltonian path in $\Gamma(G')$ with initial and terminal vertices $x_{1,\alpha_k}$ and $y_{k,1}$, respectively. As $y_{k,1}\sim x_{1,\alpha_k}$,  we get a Hamiltonian cycle in $\Gamma(G')$. This completes the proof.
\end{proof}

\begin{theorem}
    Let  $G= \mathbb{Z}_n  \times  \mathcal{G}$. Then $\c$ is Hamiltonian.
\end{theorem}
\begin{proof}
    If $\mathcal{G}$ is a $p$-group, then the result holds by Theorem~\ref{pgroup Hamiltonian}. Else, the proof is similar to the proof of Theorem~\ref{ZnPG Hamiltonian}. Hence, we omitted the proof.
\end{proof}

Now we prove that if $G= \mathbb{Z}_n  \times \mathcal{P} \times  \mathcal{G}$, where $\mathcal{P}$ is a $2$-group of maximal class, then $\c$ is Hamiltonian. The approach used in proving Theorem~\ref{P maximal hamiltonian} differs from that of Theorem \ref{ZnPG Hamiltonian} due to a key structural property of  \( \mathcal{P} \). In this case, we have \( \mathfrak{m}_{2^\alpha}(G) = 1 \) for \( \alpha \geq 3 \). This condition prevents us from directly applying the technique used in Theorem \ref{ZnPG Hamiltonian}. Consequently, a different approach is required to establish the Hamiltonicity of \( \Gamma(G) \) in this case.

\begin{theorem}{\label{P maximal hamiltonian}}
    Let $G= \mathbb{Z}_n  \times \mathcal{P} \times  \mathcal{G}$, where $\mathcal{P}$ is a $2$-group of maximal class. Then $\c$ is Hamiltonian.
\end{theorem}
\begin{proof}
    Due to Lemma~\ref{Abdollahi Hamiltonian Lemma}, it is sufficient to show that $\Gamma (G/ \cy )$ is Hamiltonian. We consider the following three cases.

 \noindent   \textbf{Case 1:} $\mathcal{P}=\mathcal{D}_{2^{t+1}}$,
    where $t\geq 1$. In this case, recall that $\cy = \mathbb{Z}_n \times \{1\} \times \{e'\} $. It follows that $G/\cy \cong \mathcal{D}_{2^{t+1}} \times  \mathcal{G}$. Let $G'=\mathcal{D}_{2^{t+1}} \times  \mathcal{G}$. If $\mathcal{G}$ is trivial, then by Theorem \ref{Abdollahi Hamiltonian propostion}, the graph $\Gamma(G')$ is Hamiltonian. 
    
    Now consider $\mathcal{G}\neq \{e'\}$. Let $P_1,P_2, \ldots , P_k$ be the Sylow subgroups of $G'$. Suppose that \linebreak[4] $\mathrm{exp}(G')=p_1^{\alpha_1}p_2^{\alpha_2} \cdots p_k^{\alpha_k}$, where $p_1>p_2>\cdots >p_k$ and $p_k^{\alpha_k}=2^t$. Let $x\in \Omega _{p_1^{\alpha_1}p_2^{\alpha_2} \cdots p_k^{\alpha_k}}(G')$ such that $( b, e')\notin \langle x\rangle$. By Lemma \ref{m_d>2 G}, we have $\mathfrak{m}_{o(g)}(G')\geq 3$ for all $g\in \{(x_1,x_2 )\in G': x_2\neq e'\}$.  As in the proof of Theorem \ref{ZnPG Hamiltonian}, we get a Hamiltonian path $H_1$  from $x$ to some $y\in \Omega_{p_{k-1}p_k^{\alpha_k}}(G')$ in the subgraph induced by $\{(x_1,x_2 )\in G': x_2\neq e'\}$. Now consider the set $\{(g,e'): g\neq 1\}$.
    Note that either $ (ab,e') \notin \langle y \rangle$ or $(a^2b,e') \notin \langle y \rangle$. Since $\langle y \rangle$ contains an element of order $2$, therefore either $\langle (ab,e'), y \rangle$ or $\langle (a^2b,e'), y \rangle$ is non-cyclic.  Without loss of generality, assume that $\langle (ab,e'), y \rangle$ is not cyclic, that is, $y\sim (ab,e')$. Also, note that the subgroup $\langle (a^ib,e'), (a^j,e')\rangle$ is non-cyclic, that is, $(a^ib,e')\sim (a^j,e') $ for $1\leq i,j \leq 2^t-1$. Thus $(ab,e')\sim (a,e')\sim (a^2b,e')\sim (a^2,e')\sim \cdots \sim (a^{2^{t}-1},e') \sim (b,e')$.  Accordingly, we get a Hamiltonian path $H_2$ in the subgraph of $\Gamma(G')$ induced by $\{(g,e'): g\neq 1\}$ with initial and terminal vertices $(ab,e')$ and  $(b,e')$, respectively. 

    Since $( b, e')\notin \langle x\rangle$, and $\langle x \rangle$ contains an element of order $2$, the group $\langle ( b, e'), x \rangle$ is non-cyclic, that is, $( b, e')\sim  x$. Thus the paths $H_1$ and $H_2$ produce a Hamiltonian path in $\Gamma (G')$.

      \noindent \textbf{Case 2:} $\mathcal{P}= \mathcal{Q}_{2^{t+1}}$, where $t\geq 2$. In this case, note that $\cy = \mathbb{Z}_n \times \{1, b^2\} \times \{e'\} $. It follows that $G/\cy \cong \mathcal{D}_{2^{t}} \times  \mathcal{G}$. By \textbf{Case 1}, we find that the graph $\Gamma(G/ \cy)$ is Hamiltonian.

\noindent \textbf{Case 3:} $\mathcal{P}=  \mathcal{SD}_{2^{t+1}}$, where $t\geq 3$. In this case, recall that $\cy = \mathbb{Z}_n \times \{1\} \times \{e'\} $, and so $G/\cy \cong \mathcal{SD}_{2^{t+1}} \times  \mathcal{G}$. Let $G'=\mathcal{SD}_{2^{t+1}} \times  \mathcal{G}$. If $\mathcal{G}$ is trivial, then by Theorem \ref{Abdollahi Hamiltonian propostion}, $\Gamma(G')$ is Hamiltonian. 

     Now consider the case that \(\mathcal{G} \neq \{e'\}\). Denote the Sylow subgroups of \(G'\) by \(P_1, P_2, \ldots, P_k\). Suppose \(\mathrm{exp}(G') = p_1^{\alpha_1}p_2^{\alpha_2} \cdots p_k^{\alpha_k}\), where \(p_1 > p_2 > \cdots > p_k\) and \(p_k^{\alpha_k} = 2^t\). Let \(x \in \Omega_{p_1^{\alpha_1}p_2^{\alpha_2} \cdots p_k^{\alpha_k}}(G')\) such that \((b, e') \notin \langle x \rangle\). By Lemma \ref{m_d>2 G}, we know that \(\mathfrak{m}_{o(g)}(G') \geq 3\) for \(g \in \{(x_1, x_2) \in G' : x_2 \neq e'\}\). 

As in the proof of Theorem \ref{ZnPG Hamiltonian}, we get a Hamiltonian path $H_1$  from $x$ to some $y\in \Omega_{p_{k-1}p_k^{\alpha_k}}(G')$ in the subgraph induced by $\{(x_1,x_2 )\in G': x_2\neq e'\}$. Now consider the set $\{(g,e'): g\neq 1\}$.

Note that \(\langle (a, e') \rangle \subseteq \langle y \rangle\), as  \(\langle (a, e') \rangle\) is the only cyclic subgroup of $G'$ of order $p_k^{\alpha_k}$. Then \((a^{2^{t-2}}, e')\), \((a^{3.2^{t-2}}, e')\) are the elements of order \(4\) contained in \(\langle y \rangle\). Since $o((ab, e'))=4$ and neither $(ab, e')=(a^{2^{t-2}}, e')$ nor $(ab, e')=(a^{3.2^{t-2}}, e')$, it follows that the subgroup \(\langle (ab, e'), y \rangle\) is non-cyclic. Thus $y\sim (ab, e')$. Further, note that 
\begin{align*}
  &  (ab, e') \sim (a, e') \sim (a^3b, e') \sim (a^3, e') \sim \cdots \sim (a^{2^{t}-1}b, e')  \sim  (a^{2^{t}-1}, e') \sim   (a^2b, e'), \text{ and } \\
  & (a^2b, e') \sim  (a^2 , e') \sim  (a^4b, e') \sim  (a^4, e')\sim \cdots \sim (a^{2^t-2},e') \sim (b, e').
\end{align*}
 Accordingly, we get a Hamiltonian path $H_2$ in the subgraph of $\Gamma(G')$ induced by $\{(g,e'): g\neq 1\}$ with initial and terminal vertices $(ab,e')$ and  $(b,e')$, respectively. 
Thus the paths $H_1$ and $H_2$ produce a Hamiltonian path with initial and terminal vertices $x$ and $(b,e')$, respectively, in $\Gamma (G')$.
Since \((b, e') \notin \langle x \rangle\) and \(\langle x \rangle\) contains an element of order \(2\), it follows that \(\langle (b, e'), x \rangle\) is non-cyclic.  Consequently, we obtain a Hamiltonian cycle
in \(\Gamma(G')\). 
\end{proof}
\section{Perfect and Total Perfect Codes in $\c$}

In this section, we characterize all finite non-cyclic groups $G$ such that the non-cyclic graph \( \Gamma(G) \) admits a perfect code. We establish that \( \Gamma(G) \) admits a perfect code if and only if \( G \) has a maximal cyclic subgroup of order \( 2 \). Finally, we prove that \( \Gamma(G) \) does not admit total perfect code when \( G \) is a finite non-cyclic nilpotent group.



\begin{proposition} {\label{perfect code propostion}}
    Let \( G \) be a finite non-cyclic group. If \( C \) is a perfect code of \( \Gamma(G) \) and \( x \in C \), then \( o(x) = 2 \).  
\end{proposition}  

\begin{proof}  
Let \( C \) be a perfect code of \( \Gamma(G) \) and \( x \in C \).    Suppose, if possible, \( o(x) \geq 3 \). Then \( x \neq x^{-1} \) and the subgroup \( \langle x, x^{-1} \rangle \) is cyclic.  Now consider a vertex \( y \in V(\c) \setminus C \) such that \( x \sim y \). Since \( \langle x, y \rangle \) is non-cyclic, it follows that \( \langle x^{-1}, y \rangle \) is also non-cyclic. Consequently,  \( x^{-1} \in V(\c) \).    Since $C$ is a perfect code, $x\in C$, $y\in V(\c)\setminus C$, and $y$ is adjacent to both $x$ and $x^{-1}$ we must have that $ x^{-1} \in V(\c) \setminus C $. In that case, there exists a unique \( z \in C \) such that \( z \sim x^{-1} \). This implies that \( \langle z, x^{-1} \rangle \) is non-cyclic. Consequently, \( \langle z, x \rangle \) is also non-cyclic, which means \( z \sim x \) in \( \c \). Thus, $z$ and $x$ are two adjacent vertices in $C$, contradicting the fact that $C$ is an independent set. Hence we must have that $o(x)=2$.  
\end{proof} 
        

\begin{theorem}{\label{perfect code theorem}}
Let $G$ be a non-cyclic finite group. Then the following conditions are equivalent 
\begin{itemize}
    \item[(i)] $\Gamma(G)$ admits a perfect code.
    \item[(ii)]  $\Gamma(G)$ has a dominating vertex.
    \item[(iii)] $G$ has a maximal cyclic subgroup of order $2$.
\end{itemize}
\end{theorem}

\begin{proof}
 Let $C$ be a perfect code of $\Gamma(G)$. Suppose, if possible, $C$ contains two distinct elements $x$ and $y$. By Proposition \ref{perfect code propostion},  $o(x) = o(y) = 2$. Since a cyclic subgroup can contain at most one element of order $2$, the subgroup $\langle x, y \rangle$ must be non-cyclic. Thus $C$ has two adjacent vertices, a contradiction. Hence $C$ must be a singleton set.
 Accordingly, the element in $C$ is a dominating vertex of $\c$. This proves that (i) implies (ii).

 Let $x\in V(\c)$ be a dominating vertex. We claim that $\langle x \rangle$ is a maximal cyclic subgroup of order $2$. This will prove that (ii) implies (iii).

On the contrary, assume that $o(x) \geq 3$. Let $y$ be an arbitrary element of $V(\Gamma(G))$. Since $x$ is a dominating vertex, it follows that $x \sim y$, which means $\langle x, y \rangle$ is non-cyclic. Consequently, $\langle x^{-1}, y \rangle$ is non-cyclic. Thus $x^{-1}\notin \cy$, that is,  $x^{-1}\in V(\c)$. However,  $\langle x, x^{-1} \rangle$ is a cyclic subgroup of $G$. Therefore $x \nsim x^{-1}$, a contradiction to the fact that $x$ is a dominating vertex. Hence $o(x) = 2$.

Now suppose that $\langle x \rangle$ is not a maximal cyclic subgroup. Then there exists an element $g_1 \in G$ such that $\langle x \rangle \subset \langle g_1 \rangle$ and $\langle g_1 \rangle \in \m$, where $\m$ denotes the collection of all maximal cyclic subgroups of $G$. Since $G$ is non-cyclic, we have $|\m| \geq 2$. Let $\langle g_2 \rangle \in \m$  such that $\langle g_1 \rangle \neq \langle g_2 \rangle$. Then $\langle g_1, g_2 \rangle$ is a non-cyclic subgroup of $G$. Consequently, $g_1 \notin \cy$, that is, $g_1 \in V(\Gamma(G))$. However, $\langle x, g_1 \rangle = \langle g_1 \rangle$, a cyclic subgroup of $G$. Therefore $x \nsim g_1$, contradicting the assumption that $x$ is a dominating vertex. Thus $\langle x \rangle$ must be a maximal cyclic subgroup of order $2$. 

Now we prove that (iii) implies (i). Let $\langle g \rangle$ be a maximal cyclic subgroup of order $2$. Let $g' \in V(\c)$ such that $g\neq g'$. If $\langle g, g' \rangle$ is cyclic, then it will be a cyclic subgroup of $G$ properly containing $\langle g \rangle$, contradicting the fact that $\langle g \rangle$ is a maximal cyclic subgroup. Therefore $\langle g, g' \rangle$ is non-cyclic, that is, $g \sim g'$. Hence $\{g\}$ is a perfect code of $\c$. This completes the proof.
\end{proof}
Theorem \ref{perfect code theorem} establishes necessary and sufficient conditions for the existence of a perfect code in the non-cyclic graph of a finite group. As a consequence, we obtain the following corollary, which specializes this result to finite non-cyclic nilpotent groups. The proof of the following corollary follows directly from Theorem \ref{perfect code theorem} and Proposition \ref{nilpotent maximal}.
\begin{corollary}
    Let $G$ be a finite non-cyclic nilpotent group. Then the non-cyclic graph $\c$ admits a perfect code if and only if $G$ is a $2$-group and it has a maximal cyclic subgroup of order $2$.
\end{corollary}
\begin{theorem}
    Let $G$ be a finite non-cyclic nilpotent group. Then $\c$ does not admit total perfect code.
\end{theorem}
\begin{proof}
    Let $G$ be a finite non-cyclic nilpotent group. If possible, assume that $\c$ has a total perfect code $T$. Since $G$ is a non-cyclic nilpotent group, $G$ has a non-cyclic Sylow subgroup $P$. Let $|P|=p^\alpha$ for some $\alpha \geq 2$. Consider $x\in P$ such that $o(x)=p$. By Theorem~\ref{Subgroups theorem}, we get $\mathfrak{m}_p(P)\geq 1+p$, and so $\mathfrak{m}_p(G)\geq 1+p$. Consider $x'\in P$ such that $o(x')=p$ and $\langle x\rangle \neq \langle x'\rangle$. Note that the subgroup $\langle x, x'\rangle$ is non-cyclic. It follows that $x\notin \cy$ and so $x\in V(\c)$. Now, let $g\in T$ such that $x\sim g$. It implies that $\langle x,g \rangle$ is non-cyclic. If $\gcd (o(x), o(g))=1$, then by Theorem~\ref{nilpotent}, we get $xg=gx$ and so $\langle x,g \rangle = \langle xg\rangle$. It follows that $\langle x,g \rangle\rangle$ is cyclic, a contradiction. Thus $\gcd (o(x), o(g))=p$. Let $y\in \langle g\rangle$ such that $o(y)=p$. Since $\langle x, g \rangle$ is non-cyclic, we must have $\langle x \rangle \neq \langle y \rangle$.  Let $h\in T$ such that $y\sim h$. Therefore $\gcd (o(y), o(h))=p$. Let $z\in \langle h \rangle$ such that $o(z)=p$. Clearly, $\langle z \rangle \neq \langle y \rangle$.

    Now, if $z\neq x$, then $\langle x, h \rangle $ is non-cyclic. It implies that $x\sim g$ and $x\sim h$, which is a contradiction to the fact that $x$ is adjacent to exactly one vertex of $T$. Suppose $z=x$. Consider $w\in G$ such that $o(w)=p$, $\langle x\rangle \neq \langle w \rangle$ and $\langle y\rangle \neq \langle w \rangle$. Then the subgroups $\langle w,g\rangle$ and $\langle w,h\rangle$ are non-cyclic. Consequently, $w\sim g$ and $w\sim h$, again a contradiction. Thus, $\c$ does not admit a total perfect code.
\end{proof}

\section{Acknowledgement}

\vspace{.5cm}
\textbf{Data Availability:} Data sharing not applicable to this article as no datasets were generated or analyzed during
the current study.
\vspace{.5cm}

\textbf{Funding}: {\sloppy The first author gratefully acknowledges the financial support received from the Indian Institute of Technology Guwahati under the Institute Post Doctoral Fellowship (IITG/R\&D/IPDF/2024-25/20240828P1181).\par}


\vspace{1cm}
\noindent
{\bf Parveen\textsuperscript{\normalfont 1}, Bikash Bhattacharjya\textsuperscript{\normalfont 1}  }
\bigskip

\noindent{\bf Addresses}:

\vspace{5pt}


\end{document}